\documentclass{amsart}[12pt]

\newtheorem{theorem}{Theorem}[section]
\newtheorem{lemma}[theorem]{Lemma}

\newtheorem{corollary}[theorem]{Corollary}

\newtheorem{problem}[theorem]{Problem}

\theoremstyle{definition}

\newtheorem{remark}[theorem]{Remark}

\usepackage{amscd,amssymb}

\begin{document}

\title[Additive primitive length]
{Additive primitive length\\
in relatively free algebras}
\author{Vesselin Drensky}
\date{}
\address{Institute of Mathematics and Informatics,
Bulgarian Academy of Sciences,
Acad. G. Bonchev Str., Block 8,
1113 Sofia, Bulgaria}
\email{drensky@math.bas.bg}
\subjclass[2010]
{13F20, 13P05, 14E07, 14R10, 17B30, 17B40.}
\keywords{primitive elements, automorphisms, polynomial algebras, free metabelian Lie algebras.}

\maketitle

\begin{abstract}
The {\it additive primitive length} of an element $f$ of a relatively free algebra $F_d({\mathfrak V})$ in a variety of algebras $\mathfrak V$
is equal to the minimal number $\ell$ such that $f$ can be presented as a sum of $\ell$ primitive elements. We give an upper bound
for the additive primitive length of the elements in the $d$-generated polynomial algebra over a field of characteristic 0, $d>1$.
The bound depends on $d$ and on the degree of the element.
We show that if the field has more than two elements, then the additive primitive length
in free $d$-generated nilpotent-by-abelian Lie algebras
is bounded by 5 for $d=3$ and by 6 for $d>3$. If the field has two elements only, then our bound are 6 for $d=3$ and 7 for $d>3$.
This generalizes a recent result of Ela Ayd{\i}n for two-generated free metabelian Lie algebras.
In all cases considered in the paper the presentation of the elements as sums of primitive can be found effectively in polynomial time.

\end{abstract}

\section*{Introduction}
Let $F_d({\mathfrak V})$ be the $d$-generated relatively free algebra in a variety of algebras $\mathfrak V$ over a field $K$.
An element $u$ in $F_d({\mathfrak V})$ is {\it primitive} if it is an image of $x_1$ under an automorphism of $F_d({\mathfrak V})$,
where $X_d=\{x_1,\ldots,x_d\}$ is a free generating system of $F_d({\mathfrak V})$.
The {\it additive primitive length} $\text{plength}_+(f)$ of an element $f$ of $F_d({\mathfrak V})$
is equal to the minimal $\ell$ such that $f$ can be presented as a sum of $\ell$ primitive elements.
If such a presentation is impossible, we define $\text{plength}_+(f)=\infty$.
The {\it additive primitive width} $\text{pwidth}_+(F_d({\mathfrak V}))$ of $F_d({\mathfrak V})$
is equal to the maximum of $\text{plength}_+(f)$, $f\in F_d({\mathfrak V})$.
The notions of additive primitive length and width were introduced by Ayd{\i}n in her recent paper \cite{A}
by analogy with the notions of primitive length and width for free groups defined and studied by Bardakov, Shpilrain, and Tolstykh \cite{BST}.
Ayd{\i}n \cite{A} calculated the additive primitive length of the elements
of the two-generated free metabelian (i.e., solvable of class 2) Lie algebra $F_2({\mathfrak A}^2)$ over a field $K$ of characteristic 0.
In particular she described the elements of infinite length in $F_2({\mathfrak A}^2)$.
Ayd{\i}n also asked the problem whether the elements of $F_d({\mathfrak A}^2)$, $d>2$, are of finite length and, if this is true,
what is $\text{pwidth}_+(F_d({\mathfrak A}^2))$ and whether it depends on $d$.

The first topic of the present paper is the study of the additive primitive length of the elements of the polynomial algebra $K[X_d]$
over a field $K$ of characteristic 0.
The case $d=1$ is trivial because the only primitive elements are of the form $\beta +\gamma x_1$, $\beta,\gamma\in K$, $\gamma\not=0$.
Hence if $\deg(f)>1$, $f\in K[X_1]$, then $\text{plength}_+(f)=\infty$. Our main result in this direction is that if $d>1$ and $\deg(f)=n>1$,
then $\displaystyle \text{plength}_+(f)\leq \binom{n+d-1}{d-1}$. In our presentation of $f$ as a sum of primitive elements
we use tame primitive elements only, i.e., images of $x_1$ under tame automorphisms of $K[X_d]$.
Unfortunately, our proof does not work when the base field is of positive characteristic.
We show that for $K=\mathbb Q$ the presentation of $f$ can be found effectively in polynomial time.
Hence the problem to find the presentation belongs to the class {\bf P}
of decision problems which can be solved by a deterministic Turing machine using a polynomial amount of computation time.
It is interesting to know whether $\text{plength}_+(f)$ is bounded
when $f\in K[X_d]$. There are some reasons to believe that, at least for $K[X_2]$, the additive primitive
length of the polynomials is not bounded, but we cannot prove this.

Then we consider the additive primitive length for the elements of the free metabelian Lie algebra $F_d({\mathfrak A}^2)$ over an arbitrary field $K$.
In the case $d=2$ Ayd{\i}n \cite{A} uses the description of the automorphisms of $F_2({\mathfrak A}^2)$ given by Shmelkin \cite{Sh}
when $\text{char}(K)=0$. But the same description of $\text{Aut}(F_2({\mathfrak A}^2))$ is true for any field $K$, see \cite{D}.
We consider the case $d>2$ and $f\in F_d({\mathfrak A}^2)$, $\deg(f)>1$. For $d=3$ we show that $\text{plength}_+(f)\leq 5$
if $K$ has more than two elements and $\text{plength}_+(f)\leq 6$ if $\vert K\vert=2$. If $d>3$, then
$\text{plength}_+(f)\leq 6$ when $\vert K\vert>2$ and $\text{plength}_+(f)\leq 7$ when $\vert K\vert=2$.
The primitive elements which appear as summands in our presentation of $f$ are images of $x_1$
under tame and inner automorphisms. Finally, we consider the relatively free algebras $F_d({\mathfrak N}_c{\mathfrak A})$ of the variety
${\mathfrak N}_c{\mathfrak A}$ of (nilpotent of class $c$)-by-abelian Lie algebras, $c>1$. As in the case of groups, see Bryant and Gupta \cite{BG},
an endomorphism of $F_d({\mathfrak N}_c{\mathfrak A})$ is an automorphism if and only if it induces an automorphism on the free metabelian Lie algebra
$F_d({\mathfrak A}^2)$, see \cite{D}. This easily implies that the additive primitive length of $f\in F_d({\mathfrak N}_c{\mathfrak A})$
is the same as of the image of $f$ in $F_d({\mathfrak A}^2)$. Again, the presentation of the elements of $F_d({\mathfrak N}_c{\mathfrak A})$
as a sum of primitive can be found effectively in polynomial time.

For a background on automorphisms of polynomial algebras, free groups, and free algebras we refer to the books by van den Essen \cite{E}
and by Mikhalev, Shpilrain, and Yu \cite{MSY}, and on varieties of Lie algebras to the book by Bahturin \cite{Bah}.

\section{Polynomial algebras}
In this section we fix a field $K$ of characteristic 0 and work in the polynomial algebra $K[X_d]$, $d>1$.
Recall that the group of {\it tame automorphisms} of $K[X_d]$ is generated by the {\it affine automorphisms} $\varphi$ defined by
\[
\varphi(x_j)=\beta_j+\sum_{i=1}^d\gamma_{ij}x_i,\quad j=1,\ldots,d,
\]
where $\beta_j,\gamma_{ij}\in K$ and the matrix $(\gamma_{ij})$ is invertible ($\varphi$ is called {\it linear} if all $\beta_j$ are equal to 0),
and the {\it triangular automorphisms} $\vartheta$ defined by
\[
\vartheta(x_j)=\gamma_jx_j+v_j(x_{j+1},\ldots,x_d),\quad j=1,\ldots,d,
\]
where $0\not=\gamma_j\in K$ and $v_j(x_{j+1},\ldots,x_d)\in K[x_{j+1},\ldots,x_d]$.

\begin{lemma}\label{lemma primitive polynomials}
Let $\alpha,\beta,\gamma,\xi_p\in K$, $p=2,\ldots,n$, $\alpha,\gamma\not=0$,
and let
\begin{equation}\label{linear element}
s(\alpha)=x_1+\sum_{i=2}^d\alpha^{(n+1)^{i-2}}x_i.
\end{equation}
Then the polynomial
\begin{equation}\label{primitive polynomial}
u=\beta+\gamma x_1+\sum_{p=2}^n\xi_ps^p(\alpha)
\end{equation}
is primitive.
\end{lemma}

\begin{proof}
We define the triangular automorphism $\vartheta$ and the affine automorphism $\varphi$ by
\[
\vartheta(x_1)=\beta+\gamma x_1+\sum_{p=2}^n\xi_px_2^p,\vartheta(x_j)=x_j,\quad j=2,\ldots,d,
\]
\[
\varphi(x_1)=x_1,\varphi(x_2)=s(\alpha)=x_1+\sum_{i=2}^d\alpha^{(n+1)^{i-2}}x_i,\varphi(x_j)=x_j,\quad j=3,\ldots,d.
\]
Since $u=\varphi\circ\vartheta(x_1)$, it is primitive.
\end{proof}

The following theorem is the first main result in our paper.

\begin{theorem}\label{plength in polynomial algebras}
Let the base field $K$ be of characteristic $0$ and let $d>1$.
If $f\in K[X_d]$ and $\deg(f)=n>1$,
then
\[
\text{\rm plength}_+(f)\leq \binom{n+d-1}{d-1}.
\]
\end{theorem}

\begin{proof} We write $f$ in the form
\[
f=f_0+f_1+\cdots+f_n,
\]
where
\[
f_p=\sum_{\vert a\vert=p} \mu_ax_1^{a_1}\cdots x_d^{a_d},\quad \mu_a\in K,\vert a\vert=a_1+\cdots+a_d,
\]
is the homogeneous component of $f$ of degree $p$.
If $f_1\not=0$ there is a linear automorphism $\psi$ which sends $f_1$ to $x_1$. Since the automorphisms do not change the additive primitive length
of the elements, we can work with $\psi(f)$ instead with $f$ and may assume that
$f=\beta+x_1+f_2+\cdots+f_n$, $\beta\in K$.
If $f_1=0$, then $f=\beta+f_2+\cdots+f_n$. Hence in both cases we can work with $f$ presented in the form
\[
f=\beta+\delta x_1+f_2+\cdots+f_n,\quad \beta\in K,\delta=0,1.
\]
Since the base field is of characteristic 0 we can choose
$\displaystyle N=\binom{n+d-1}{d-1}$ pairwise different nonzero elements $\alpha_1,\ldots,\alpha_N\in K$
such that for a fixed $k$ all powers $1,\alpha_k,\alpha_k^2,\ldots,\alpha_k^{n(n+1)^{d-2}}$ are pairwise different.
Additionally, we require that the matrices
\begin{equation}\label{Schur matrix}
\left(\begin{matrix}
\alpha_1^{p_1}&\ldots&\alpha_N^{p_1}\\
\vdots&\ddots&\vdots\\
\alpha_1^{p_N}&\ldots&\alpha_N^{p_N}\\
\end{matrix}\right)
\end{equation}
are invertible for all nonnegative integers $p_1<\cdots<p_N\leq n(n+1)^{d-2}$.
Such matrices appear in the determinantal presentation of Schur functions, see e.g. \cite{Mc}.
We want to present $f$ as a sum of the primitive elements $u_1,\ldots,u_N$ of the form (\ref{primitive polynomial}), where
\[
u_1=\beta+\xi_{11}x_1+\sum_{p=2}^n\xi_{1p}s^p(\alpha_1),
\]
\[
u_k=\xi_{k1}x_1+\sum_{p=2}^n\xi_{kp}s^p(\alpha_k),\quad k=2,\ldots,N,
\]
and $s(\alpha_k)$ is defined in (\ref{linear element}) for $\alpha=\alpha_k$.
Here the coefficients $\xi_{kp}$ are unknown elements of $K$.
We can choose all $\xi_{k1}\in K$ in such a way that they are different from 0 and $\xi_{11}+\cdots+\xi_{N1}=\delta$,
i.e., the affine component $\beta+\delta x_1$ of $f$ is equal to the affine component of $u_1+\cdots+u_N$.
By Lemma \ref{lemma primitive polynomials} the polynomials $u_1,\ldots,u_N$ are primitive.
Now, let $1<p\leq n$. Then for the homogeneous component $f_p$ of degree $p$ of $f$ we obtain
\[
f_p=\sum_{\vert a\vert=p}\mu_ax_1^{a_1}\cdots x_d^{a_d}
=\sum_{k=1}^N\xi_{kp}s^p(\alpha_k)
\]
\[
=\sum_{k=1}^N\xi_{kp}\sum_{\vert a\vert=p}\binom{a_1+\cdots+a_d}{a_1,\ldots,a_d}\alpha_k^{P_a}x_1^{a_1}\cdots x_d^{a_d},
\]
where
\begin{equation}\label{numerical system}
P_a=a_1+a_2(n+1)+a_3(n+1)^2+\cdots+a_d(n+1)^{d-2}
\end{equation}
and $\displaystyle \binom{a_1+\cdots+a_d}{a_1,\ldots,a_d}=\frac{(a_1+\cdots+a_d)!}{a_1!\cdots a_d!}$
is the multinomial coefficient. Comparing the coefficients of the $\displaystyle N_p=\binom{p+d-1}{d-1}$ monomials $x_1^{a_1}\cdots x_d^{a_d}$,
we obtain the linear system
\begin{equation}\label{linear system}
\binom{a_1+\cdots+a_d}{a_1,\ldots,a_d}\sum_{k=1}^N\xi_{kp}\alpha_k^{P_a}=\mu_a,\quad \vert a\vert=p,
\end{equation}
with $N_p$ equations and $\displaystyle N=\binom{n+d-1}{d-1}$ unknowns.
Since $a_j\leq p=\deg(f_p)\leq\deg(f)=n$, $j=1,\ldots,d$, the integers $P_a$ from (\ref{numerical system})
have presentations in the numerical system with base $n+1$
and are pairwise different. Hence the rows of the matrix of the system (\ref{linear system}) are rows
of a matrix of the form (\ref{Schur matrix}). Since the matrix (\ref{Schur matrix}) is invertible, its rows are linearly independent and the system
(\ref{linear system}) has a solution. This implies that $f$ can be presented as a sum of $N$ primitive polynomials and
$\displaystyle \text{plength}_+(f)\leq\binom{n+d-1}{d-1}$.
\end{proof}

\begin{corollary}\label{complexity for polynomials}
If $f\in {\mathbb Q}[X_d]$, $\deg(f)>1$, we can find effectively a presentation of $f$ as a sum of primitive polynomials in polynomial time.
\end{corollary}

\begin{proof}
Following the proof of Theorem \ref{plength in polynomial algebras}, we can find the presentation of $f$, $\deg(f)=n>1$, as a sum of primitive polynomials
in three steps:

(1) We find a linear automorphism $\psi$ of ${\mathbb Q}[X_d]$ which send $f_1$ to $\delta x_1$;

(2) We compute $\psi(f)$;

(3) We solve $n-1$ linear systems with $\displaystyle \leq N=\binom{n+d-1}{d-1}$ equations and $N$ unknowns.

Obviously steps (1) and (2) can be performed in polynomial time. In the third step we shall solve the systems
applying the Gaussian elimination. It is well known that for a system with $\leq N$ equations and $N$ unknowns
we need ${\mathcal O}(N^3)$ arithmetic operations. The number of arithmetic operations
measures the computational complexity when the time for each arithmetic operation is approximately the same.
This happens when the coefficients of the system are represented by floating-point numbers or when we work in a finite field.
For our purposes we need to work with rational numbers represented exactly.
In this case the intermediate entries may become exponentially large and the bit complexity is exponential.
But there is a variant of Gaussian elimination due to Bareiss \cite{Bar} which avoids the exponential growth of the intermediate entries.
It has the same arithmetic complexity ${\mathcal O}(N^3)$, but its bit complexity is ${\mathcal O}(N^5)$.
\end{proof}

\begin{remark}
(1) In Theorem \ref{plength in polynomial algebras} we consider the case $\deg(f)>1$ because the case $\deg(f)\leq 1$ is trivial.
If $f\in K[X_d]$ and $\deg(f)=1$, then $f$ is primitive and $\text{plength}_+(f)=1$; if $\deg(f)=0$,
then $f=\beta\in K$, $f=(\beta+x_1)+(-x_1)$ and $\text{plength}_+(f)=2$.

(2) The proof of Theorem \ref{plength in polynomial algebras} does not hold for fields of positive characteristic $p$
because some of the multinomial coefficients become 0 modulo $p$.
\end{remark}

\section{Metabelian Lie algebras}
In this section we consider the free metabelian Lie algebra $F_d({\mathfrak A}^2)$, $d>2$, over an arbitrary field $K$.
It is well known, see e.g., \cite{Bah}, that as a vector space the commutator ideal $F_d'({\mathfrak A}^2)$ has a basis
\[
[x_{i_1},x_{i_2},x_{i_3},\ldots,x_{i_n}]=[x_{i_1},x_{i_2}]\text{ad}(x_{i_3})\cdots\text{ad}(x_{i_n}),
\]
where $i_1>i_2\leq i_3\leq\cdots\leq i_n$, $[v_1,v_2,\ldots,v_{n-1},v_n]=[[v_1,v_2,\ldots,v_{n-1}],v_n]$, and $v\text{ad}(w)=[v,w]$.
For the Lie algebra $F_d({\mathfrak A}^2)$ the group of tame automorphisms is generated by the linear and the triangular automorphisms
which are defined in the same way as in the case of polynomial algebras.
Additionally, we shall consider the {\it inner automorphisms} of $F_d({\mathfrak A}^2)$ defined by
\[
\exp(\text{ad}v):x_j\to x_j+[x_j,v],\quad j=1,\ldots,d,v\in F_d'({\mathfrak A}^2).
\]
(By the theorem of Bahturin and Nabiyev \cite{BN} the inner automorphisms of $F_d({\mathfrak A}^2)$ are {\it wild}, i.e., not tame.)

\begin{lemma}\label{primitive Lie elements}
The following elements of $F_d({\mathfrak A}^2)$ are primitive:
\begin{equation}\label{triangular primitive}
u_1=\gamma x_1+v(x_2,\ldots,x_d),\quad 0\not=\gamma\in K, v(x_2,\ldots,x_d)\in F_d({\mathfrak A}^2);
\end{equation}
\begin{equation}\label{inner primitive}
u_2=\gamma x_j+[w,x_j],\quad n\geq 2, 0\not=\gamma\in K, w\in F_d'({\mathfrak A}^2);
\end{equation}
\begin{equation}\label{quadratic primitive}
u_3=\sum_{j=1}^d\alpha_jx_j+\sum_{j=2}^d\beta_j[x_j,x_1],\quad \alpha_j,\beta_j\in K,
\end{equation}
if the elements $\displaystyle \sum_{j=2}^d\alpha_jx_j, \sum_{j=2}^d\beta_jx_j$, and $x_1$ are linearly independent.
\end{lemma}

\begin{proof}
The elements $u_1$ from (\ref{triangular primitive}) and $u_2$ from (\ref{inner primitive}) are primitive:
$u_1$ is the image of $x_1$ under a triangular automorphism and $u_2$ is an image of $x_j$ under the automorphism
$\displaystyle \gamma\exp\left(-\frac{1}{\gamma}\text{ad}w\right)$.
Finally, since the elements
\[
y_1=\sum_{j=1}^d\alpha_jx_j=\alpha_1x_1+\sum_{j=2}^d\alpha_jx_j, y_2=\sum_{j=2}^d\beta_jx_j,y_3=x_1
\]
are linearly independent we can extend them with elements $y_4,\ldots,y_d$ to a basis $Y_d=\{y_1,\ldots,y_d\}$ of the vector space
spanned by $X_d$. Then $u_3$ from (\ref{quadratic primitive}) has the form $u_3=y_1+[y_2,y_3]$. In this way $u_3$ is an image of $y_1$ under a triangular
automorphism with respect to the free generating system $Y_d$ of $F_d({\mathfrak A}^2)$ and hence is primitive.
\end{proof}

\begin{theorem}\label{plength in metabelian Lie algebras}
Let $f\in F_d({\mathfrak A}^2)$. Then for $d=3$
\[
\text{\rm plength}_+(f)\begin{cases}
\leq 5,\text{ if }\vert K\vert>2,\\
\leq 6, \text{ if }\vert K\vert=2.\\
\end{cases}
\]
When $d>3$, then
\[
\text{\rm plength}_+(f)\begin{cases}
\leq 6,\text{ if }\vert K\vert>2,\\
\leq 7, \text{ if }\vert K\vert=2.\\
\end{cases}
\]
\end{theorem}

\begin{proof}
As in the polynomial case in Theorem \ref{plength in polynomial algebras} we can apply a linear automorphism to bring $f$ in the form
\begin{equation}\label{simple form of f}
f=\delta x_1+\sum_{j=2}^d\beta_j[x_j,x_1]+\sum\mu_i[x_{i_1},x_1,x_{i_3},\ldots,x_{i_n}]+v(x_2,\ldots,x_d),
\end{equation}
where $\delta=0,1$, $\beta_j,\mu_i\in K$, $i_1>1\leq i_3\leq\cdots\leq i_n$, $n\geq 3$, $v(x_2,\ldots,x_d)\in F_d'({\mathfrak A}^2)$.

First, let $d=3$. Then in (\ref{simple form of f}) we present $\sum\mu_i[x_{i_1},x_1,x_{i_3},\ldots,x_{i_n}]$ as a sum of three parts
\[
[w_1,x_1]=\sum\mu_{i}^{(1)}[x_{i_1},x_1,x_{i_3},\ldots,x_{i_{n-1}},x_1],
\]
\[
[w_2,x_2]=\sum\mu_{i}^{(2)}[x_{i_1},x_1,x_{i_3},\ldots,x_{i_{n-1}},x_2],
\]
\[
[w_3,x_3]=\sum\mu_{i}^{(3)}[x_{i_1},x_1,x_{i_3},\ldots,x_{i_{n-1}},x_3].
\]
In this way $f$ has the form
\[
f=\delta x_1+(\beta_2[x_2,x_1]+\beta_3[x_3,x_1])+[w_1,x_1]+[w_2,x_2]+[w_3,x_3]+v(x_2,x_3).
\]
We want to present it as a sum of the following five primitive elements
\[
u_1=\xi x_1+v(x_2,x_3),
\]
\[
u_2^{(1)}=\xi_1 x_1+[w_1,x_1],u_2^{(2)}=\xi_2 x_2+[w_2,x_2],u_2^{(3)}=\xi_3 x_3+[w_3,x_3],
\]
\[
u_3=(\zeta_1 x_1+\zeta_2 x_2+\zeta_3 x_3)+(\beta_2[x_2,x_1]+\beta_3[x_3,x_1]),
\]
with unknown coefficients $\xi,\xi_1,\xi_2,\xi_3,\zeta_1,\zeta_2,\zeta_3\in K$.
We require $\xi,\xi_1,\xi_2,\xi_3,\zeta_2,\zeta_3\not=0$.
(If $v(x_2,x_3)$ or some $w_i$ is equal to zero, then $u_1$ or $u_2^{(i)}$ is still a primitive
element.) If $\beta_2x_2+\beta_3x_3\not=0$,
then we need the linear independence of $\zeta_2 x_2+\zeta_3 x_3$ and $\beta_2x_2+\beta_3x_3$.
If $\beta_2x_2+\beta_3x_3=0$, then $\zeta_1 x_1+\zeta_2 x_2+\zeta_3 x_3$ may be any. In the latter case, if $\zeta_1=\zeta_2=\zeta_3=0$, then
simply $u_3$ does not participate in the presentation of $f$.
The components of $f$ and $u_1+u_2^{(1)}+u_2^{(2)}+u_2^{(3)}+u_3$ in $F_3'({\mathfrak A}^2)$ coincide.
Comparing the linear components we obtain the system
\begin{equation}\label{system for d=3}
\xi+\xi_1+\xi_2+\xi_3+\zeta_1=\delta,\xi_2+\zeta_2=0,\xi_3+\zeta_3=0.
\end{equation}
This system always has a solution of nonzero $\xi,\xi_1,\xi_2,\xi_3,\zeta_2,\zeta_3$ and some $\zeta_1$.
When $\beta_2x_2+\beta_3x_3\not=0$ and the field $K$ has more than two elements we can vary $\zeta_2,\zeta_3$ to guarantee
the linear independence of $\zeta_2x_2+\zeta_3x_3$ and $\beta_2x_2+\beta_3x_3$. Hence $\text{plength}_+(f)\leq 5$.
If $\vert K\vert=2$ we can add one more primitive summand: When the only solution for $\beta_2$ and $\beta_3$ of the system (\ref{system for d=3})
implies that $\zeta_2x_2+\zeta_3x_3$ and $\beta_2x_2+\beta_3x_3$ coincide, we split $u_3$ in two parts $u_3=u_3'+u_3''$, where
\[
u_3'=(\zeta_1 x_1+\zeta_2' x_2+\zeta_3' x_3)+(\beta_2[x_2,x_1]+\beta_3[x_3,x_1]),u_3''=\zeta_2'' x_2+\zeta_3'' x_3,
\]
$\zeta_2'+\zeta_2''=\zeta_2$, $\zeta_3'+\zeta_3''=\zeta_3$, to make $\zeta_2'x_2+\zeta_3'x_3$ and $\beta_2x_2+\beta_3x_3$ different which
gives that $u_3'$ is primitive.

Now, let $d>3$. In (\ref{simple form of f}) we present $\sum\mu_i[x_{i_1},x_1,x_{i_3},\ldots,x_{i_n}]$ as a sum of four parts
\[
[w_1,x_d],\quad w_1=\sum\mu_i^{(1)}[x_d,x_1,x_{i_3},\ldots,x_{i_{n-1}}],
\]
\[
[w_2,x_{d-1}],\quad w_2=\sum\mu_i^{(2)}[x_d,x_1,x_{i_3},\ldots,x_{i_{n-1}}],
\]
\[
w_3=\sum\mu_i^{(3)}[x_d,x_1,x_{i_3},\ldots,x_{i_n}],\quad i_n<d-1,
\]
\[
w_4=\sum\mu_i^{(4)}[x_{i_1},x_1,x_{i_3},\ldots,x_{i_n}],\quad i_1<d,i_n<d-1.
\]
We want to present $f$ as a sum of six primitive elements
\[
u_1^{(1)}=\xi x_1+v(x_2,\ldots,x_d),u_1^{(2)}=\eta_{d-1}x_{d-1}+w_3,u_1^{(3)}=\eta_dx_d+w_4,
\]
\[
u_2^{(1)}=\xi_dx_d+[w_1,x_d],u_2^{(2)}=\xi_{d-1}x_{d-1}+[w_2,x_{d-1}],
\]
\[
u_3=(\zeta_1 x_1+\zeta_{d-1} x_{d-1}+\zeta_d x_d)+\sum_{j=2}^d\beta_j[x_j,x_1],
\]
where $\xi,\eta_{d-1},\eta_d,\xi_{d-1},\xi_d\not=0$ and if some $\beta_2,\ldots,\beta_d$ is different from 0, then
$\zeta_{d-1} x_{d-1}+\zeta_dx_d$ and $\displaystyle \sum_{j=2}^d\beta_jx_j$ are linearly independent.
(If $\beta_2=\cdots=\beta_d=0$, then it is also allowed $u_3=0$.)
Since $i_1>1\leq i_2\leq\cdots\leq i_n$, the generator $x_{d-1}$ does not participate in the expression of $w_3$.
Similarly, $x_d$ does not participate in $w_4$.
By Lemma \ref{primitive Lie elements} all summands $u_1^{(1)},u_1^{(2)},u_1^{(3)},u_2^{(1)},u_2^{(2)},u_3$ are primitive.
We still have to arrange the equality of the linear components. This means to solve the system
\begin{equation}\label{system for d>3}
\xi+\zeta_1=\delta,\quad \eta_{d-1}+\xi_{d-1}+\zeta_{d-1}=0,\quad \eta_d+\xi_d+\zeta_d=0
\end{equation}
obtained comparing the coefficients of $x_1,x_{d-1},x_d$. As in the case of polynomial algebras when $\vert K\vert>2$ we can find
a solution of the system (\ref{system for d>3}) which satisfies all the requirements for being different from 0 and for linear independence.
If $\vert K\vert=2$ we have to add one more linear primitive element to arrange the required restrictions.
\end{proof}

\begin{remark}
As in Corollary \ref{complexity for polynomials}, when $K=\mathbb Q$ in Theorem \ref{plength in metabelian Lie algebras}
we can find the presentation of $f\in F_d({\mathfrak A}^2)$ as a sum of primitive elements in polynomial time:
Repeating the steps of the proof of Corollary \ref{complexity for polynomials} we have to perform steps (1) and (2) only
and then to solve the system (\ref{system for d=3}) when $d=3$ or (\ref{system for d>3}) when $d>3$.
\end{remark}

\begin{corollary}\label{nilpotent-by-abelian Lie algebras}
Let $F_d({\mathfrak N}_c{\mathfrak A})$ be the relatively free algebra of rank $d>1$ of the variety
${\mathfrak N}_c{\mathfrak A}$ of {\rm (}nilpotent of class $c${\rm )}-by-abelian Lie algebras, $c>1$.
Then the additive primitive length of $f\in F_d({\mathfrak N}_c{\mathfrak A})$
is the same as of the image of $f$ in $F_d({\mathfrak A}^2)$.
\end{corollary}

\begin{proof}
There is a canonical epimorphism $\pi:F_d({\mathfrak N}_c{\mathfrak A})\to F_d({\mathfrak A}^2)$.
Since an endomorphism of $F_d({\mathfrak N}_c{\mathfrak A})$ is an automorphism if and only if it induces an automorphism on
$F_d({\mathfrak A}^2)$, see \cite{D}, we obtain that $f\in F_d({\mathfrak N}_c{\mathfrak A})$ is primitive if and only if its image
$\pi(f)$ is primitive in $F_d({\mathfrak A}^2)$. Hence, if $\text{plength}_+(\pi(f))=\infty$ in $F_2({\mathfrak A}^2)$, then
$\text{plength}_+(f)=\infty$ in $F_2({\mathfrak N}_c{\mathfrak A})$. Now, let $\text{plength}_+(\pi(f))=\ell$ in $F_d({\mathfrak A}^2)$, $d>1$,
and $\pi(f)=\pi(u_1)+\cdots+\pi(u_{\ell})$, where $u_1,\ldots,u_{\ell}\in F_d({\mathfrak N}_c{\mathfrak A})$ and
$\pi(u_1),\ldots,\pi(u_{\ell})$ are primitive in $F_d({\mathfrak A}^2)$. Hence $u_1,\ldots,u_{\ell}$ are primitive in
$F_d({\mathfrak N}_c{\mathfrak A})$. Clearly,
\[
w=f-(u_1+\cdots+u_{\ell})\in F_d''({\mathfrak N}_c{\mathfrak A})
\]
and hence $u_1'=u_1+w$ is also primitive in $F_d({\mathfrak N}_c{\mathfrak A})$. In this way
$f=u_1'+u_2+\cdots+u_{\ell}$ and $\text{plength}_+(f)=\ell=\text{plength}_+(\pi(f))$.
\end{proof}

\section{Open problems}
In this section we shall collect some open problems concerning the additive primitive length for different relatively free algebras.

\begin{problem}
Is there an upper bound for $\text{\rm plength}_+(f)$, $f\in K[X_d]$, $d\geq 2$? In other words, does  $\text{\rm pwidth}_+(K[X_d])$ exist?
What happens when $\text{\rm char}(K)>0$?
\end{problem}

By a result of Shpilrain and Yu \cite{SY} the primitive elements of the free associative algebra $K\langle X_2\rangle$ are {\it palindromic},
i.e., they coincide with their images under the antiautomorphism of $K\langle X_2\rangle$ to the opposite algebra $K\langle X_2\rangle^{\text{opp}}$.
In other words they are the same if we read their monomials backwards, from right to the left.
Hence elements of $K\langle X_2\rangle$ which are not palindromic cannot be presented as sums of primitive.

\begin{problem}
Describe the elements of $K\langle X_2\rangle$ which can be presented as sums of primitive.
\end{problem}

Since the automorphism groups of $K[X_2]$ and $K\langle X_2\rangle$ are isomorphic, one can lift in a unique way every
automorphism of $K[X_2]$ to an automorphism of $K\langle X_2\rangle$, see the comments in \cite{SY}. Hence every decomposition
of $f\in K[X_2]$ into a sum of primitive polynomials can be lifted uniquely to a similar sum in $K\langle X_2\rangle$.
The main difficulty in the above problem is that the presentation of the elements of $K[X_2]$ as a sum of primitive is not unique.

\begin{problem}
What happens with the presentation into a sum of primitive in $K\langle X_d\rangle$, $d>2$?
\end{problem}

We may ask the same question for the free Lie algebra $L_d$. Since all automorphisms of $L_2$ are linear, the interesting case is $d>2$.

Let ${\mathfrak T}_c={\mathfrak N}_c\mathfrak A$ be the variety of associative algebras
defined by the polynomial identity $[x_1,x_2]\cdots[x_{2c-1},x_{2c}]=0$.
By a theorem of Maltsev \cite{M} in characteristic 0 and of Siderov \cite{S} and other authors for infinite fields of positive characteristic
this variety is generated by the algebra of $c\times c$ upper triangular matrices. As in the case of Lie algebras, see \cite{D},
the problem for the presentation of the elements of $F_d({\mathfrak T}_c)$, $c>2$, is reduced to the same problem for $F_d({\mathfrak T}_2)$.

\begin{problem}
Describe the elements of the relatively free associative algebra $F_d({\mathfrak T}_2)$ which can be presented as sums of primitive.
\end{problem}

Applying the canonical epimorphism $\pi:F_d({\mathfrak T}_2)\to K[X_d]$, when $\text{char}(K)=0$,
we can present $\pi(f)$, $f\in F_d({\mathfrak T}_2)$,
as a sum $\pi(f)=\pi(u_1)+\cdots+\pi(u_{\ell})$ of tame primitive elements and lift the presentation to $F_d({\mathfrak T}_2)$.
Hence the problem is reduced to the problem to  find a presentation as a sum of primitive elements
of the element $f-(u_1+\cdots+u_{\ell})$ in the commutator ideal of
$F_d({\mathfrak T}_2)$. This ideal has a basis
\begin{equation}\label{associative commutator elements}
x_1^{k_1}\cdots x_d^{k_d}[x_{i_1},x_{i_2},\ldots,x_{i_n}],\quad k_1,\ldots,k_d\geq 0, i_1>i_2\leq\cdots\leq i_n,n\geq 2,
\end{equation}
see e.g., \cite{D2}. When $n\geq 3$ we can handle the elements (\ref{associative commutator elements}) as in the case of metabelian Lie algebras
and the main difficulty is to find a presentation for the elements from (\ref{associative commutator elements}) when $n=2$.

\begin{problem}
One can ask similar questions for the additive primitive length for other relatively free algebras $F_d({\mathfrak V})$,
e.g., for the varieties of associative and Lie algebras generated by the $2\times 2$ matrix algebra, or for varieties of Jordan algebras.
\end{problem}

\end{document}